\documentclass[11pt,a4paper]{article}
\usepackage{authblk}
\usepackage{indentfirst}
\usepackage{graphicx}
\usepackage{epsfig}

\usepackage{latexsym}
\usepackage{amsmath}
\usepackage{amssymb}
\usepackage{amsfonts}
\usepackage{mathrsfs}
\usepackage{pifont}

\usepackage{amsthm}
\usepackage{float}
\usepackage{subfigure}
\usepackage{color}
\usepackage{multicol}

\usepackage[colorlinks,linkcolor=blue,filecolor=black,citecolor=blue]{hyperref}
\usepackage{ulem}

\newtheorem{corollary}{Corollary}[section]
\newtheorem{definition}{Definition}[section]
\newtheorem{example}{Example}[section]
\newtheorem{theorem}{Theorem}[section]
\newtheorem{lemma}{Lemma}[section]

\newcommand{\mbf}{\boldsymbol}

\newcommand{\www}{\widetilde}

\newcommand{\pl}{\partial}

\newcommand{\OOO}{\Omega}

\title{Harmonic Measures and Numerical Computation \\
of Cauchy Problems for Laplace Equations}

\author[1]{Yu Chen }
\author[2]{Jin Cheng \thanks{Corresponding author: jcheng@fudan.edu.cn}}
\author[2]{Shuai Lu}
\author[3]{Masahiro Yamamoto}

\affil[1]{{\small School of Mathematics, Shanghai University of Finance and Economics, Shanghai, 200433, China}}
\affil[2]{{\small School of Mathematical Sciences, Fudan University, Shanghai, 200433, China}}
\affil[3]{{\small School of Mathematical Sciences, the University of Tokyo, Tokyo 153, Japan} }
\date{\today}
\begin{document}
\maketitle

\section*{Abstract}
%We consider the Cauchy problem for the Laplace equation in two dimensional
%domain.  The problem is ill-posed and we are concerned with 
%conditional stability under some a priori bound assumption on solutions.
%Our conditional stability relies on harmonic measures for given 
%measurement subboundary.
%Based on conditional stability, we establish error estimates for 
%discretized regularization and present numerical examples to 
%demonstrate the effectiveness and reliability of our analysis.

%It is well known that Cauchy problem for Laplace equations is ill-posed in Hadamard's sense. Small deviations in Cauchy data may lead to large errors in the solutions. It is observed that conditional stability can be restored under some a priori bound assumption on solutions, which gives a reasonable way to construct stable algorithms. However, it is impossible to have good results everywhere in the domain. Then it becomes an issue that how to evaluate numerical solutions and determine where the results are reliable. We show that the conditional stability relies quantitatively on harmonic measures for given measurement subboundary in two dimensional cases. Based on such indicate functions we establish point-wise error estimates of numerical reconstructions. We further propose how to determine a reliable sub-domain, in which the numerical solutions can have order of convergence rate beyond a certain threshold. 

It is well known that Cauchy problem for Laplace equations is an ill-posed problem in Hadamard's sense. Small deviations in Cauchy data may lead to large errors in the solutions. It is observed that if a bound is imposed on the solution, there exists a conditional stability estimate. This gives a reasonable way to construct stable algorithms. However, it is impossible to have good results at all points in the domain. Although numerical methods for Cauchy problems for Laplace equations have been widely studied for quite a long time, there are still some unclear points, for example, how to evaluate the numerical solutions, which means whether we can approximate the Cauchy data well and keep the bound of the solution, and at which points the numerical results are reliable? In this paper, we will prove the conditional stability estimate which is quantitatively related to harmonic measures. The harmonic measure can be used as an indicate function to pointwisely evaluate the numerical result, which further enables us to find a reliable subdomain where the local convergence rate is higher than a certain order.% Especially, for a given domain and the part of the boundary with Cauchy data, we can propose a reliable sub-domain, in which the numerical solutions can have order of convergence rate higher than a threshold. %Some numerical examples are presented to show the effectiveness and reliability of our analysis.

%$\quad$

\noindent {\bf Key words}: Conditional stability, Cauchy problem, 
Laplace equation, indicate function %unique continuation, Green function

\section{Introduction}

The Cauchy problem for the Laplace equation is a classical problem and has 
a long history (e.g., \cite{Alessandrini2009}).  
The study of Cauchy problem is of fundamental significance both theoretically 
and practically (\cite{Payne1960}). However, the numerical treatment is usually challenging, caused by the well-known ill-posedness in Hadamard's sense \cite{Hadamard1902}. Small changes in Cauchy data may lead to large deviations in the solution due to the instability of the problem.

The stability may be restored by introducing some conditions on the solutions. It is observed that, if a bound is imposed on the solution, then 
we can prove a conditional stability estimate assuming a priori boundness 
condition on solutions. 
This will give a reasonable way to construct a stable algorithm for solving the Cauchy problem for the Laplace equation by the Tikhonov regularization. The conditional stability estimates imply the convergence rate of the regularized solution \cite{Cheng2000-1}. However, it is impossible to have reasonably 
accurate results everywhere over the domain, 
even if we can approximate the Cauchy data well and keep 
the bound of the solution. Then there raises the issue at which points the numerical results are reliable, i.e, how to evaluate the numerical solutions. 
This is crucial for real applications such as remote measurement problems and design problems. 
Theoretically, by a Carleman estimate, usually a qualitative conditional 
stability can be obtained (e.g., \cite{Isakov2006}), whereas a classical
quantitative estimate in the three-circle form needs to be adapted to general geometries. In some works involving pointwise estimate, it is not direct to obtain 
the stability index function (e.g., \cite{Payne1960}). In real applications, it would be very useful if a pointwise estimate like
\[
|u(x)|\leq C\varepsilon^{\tau(x)}
\]
is available, 
where $\varepsilon$ is the observation error in a certain norm, and
$\tau(x)$ is a function that is convenient to evaluate, because one can evaluate where the reconstruction is reliable based on the convergence rate $\tau(x)$.

Various numerical algorithms have been developed to deal with the Cauchy problems. For example, the methods through Maz'ya iterative algorithm \cite{Kozlov1990} based on weak form \cite{Johansson2004} and regularized boundary element 
method (BEM) (e.g., \cite{Hrycak2004}, \cite{Yang2005}), 
the moment method \cite{Cheng2001}, through solving optimal control problem based on finite element method (\cite{Burman2018}, \cite{Chakib2006}), 
and a more common approach by Tikhonov regularization 
involving modifications of the operators of the problem 
(e.g., \cite{Lattes1969}).  Besides the convergence and stability of used 
methods, relatively less studied is the evaluation of the reconstructed solution.

In this paper, we will discuss the Cauchy problem for the Laplace equation
and prove the conditional stability estimate, in which the order function 
can be given in the form of the harmonic measure. 
The explicit expression of the order function in stability estimate can be used for estimation of discretized solutions to the Cauchy problem.
In \cite{Na}, general treatments are described for such estimation 
for discretized Tikhonov regularized solutions, and we discuss more details
limited to the Cauchy problem.  
By such an indicate function, when the reconstruction domain and the part of boundary with Cauchy data are given, we can propose a trustable sub-domain, in which the numerical solutions can have order of convergence rate greater than 
$1/2$ for example.

This paper is organized as follows: we will formulate the problem and discuss the conditional stability of the problem in Section 2. 
The numerical scheme and the related analysis are presented in Section 3, and 
error estimates are proved for discretised regularization scheme in Section 4. 
In Section 5, some examples are given to illustrate the numerical method. We present some remarks and conclusions finally in Section 6.

\section{Conditional Stability}
Let $\Omega \subset \mathbb{R}^2$ be a bounded domain with smooth
boundary $\partial\Omega$.
We consider the following Cauchy problem,
\begin{align*}
\Delta u = & 0, \quad \text{in} \quad \Omega\\
u = & f,\quad \text{on}\quad \Gamma\subset\partial\Omega\\
\partial_{\nu} u = & g ,\quad \text{on}\quad\Gamma\subset\partial\Omega ,
\end{align*}
where $\Gamma\subset\partial \Omega$ is an open subset of the boundary $\partial\Omega$, $\nu$ denotes the outer normal vector to $\partial\Omega$
and $\partial_{\nu}u:= \nabla u\cdot\nu$.
We usually assume that the mathematical models should well describe the real problems, which implies the existence of solutions, but measurement errors may 
disturb stable construction of approximatng solutions.

%The uniqueness of the solution indicates that the information from the measurements is enough to determine the solution uniquely.

In general, the Cauchy problem is unstable and causes difficulties in 
numerical treatments.  For example, we consider the following example,
\[
\Phi_n(x_1,x_2)=\frac{1}{n}(\sin nx_1) \exp (nx_2),
\]
which is harmonic in $x:= (x_1,x _2) \in \mathbb{R}^2$. 
When $\Gamma = \{ x_2=b\}$ with some $b<0$, the Cauchy data will be small but the solutions increase drastically as $x_2$ increases. This illustrates that small errors in data may probably 
be enlarged for the numerical solutions.
Then one turns to seek conditional stability results. If we can prove
conditional stability, then we can construct stable algorithms. The conditional stability estimates imply the convergence rate of the regularized solution \cite{Cheng2000-1}.
We understand conditional stability as follows.
\begin{definition}[Conditional stability \cite{Cheng2000-1}]
Let $K$ be a densely defined injective operator from a Banach space $X$ 
to a Banach space $Y$, and $\omega:\{\xi\ge 0\} \,\longrightarrow \, 
\{\xi \ge 0\}$ is a monotone increasing continuous function satisfying
$\omega(0) = 0$.  Moreover $Z\subset X$ is assumed to 
be continuously embedded in $X$ and $Q\subset Z$.
Then we say 
that in the operator equation $Kf=g$, the conditional stability holds
if, for a given $M>0$, there exists a constant $C=C(M)>0$ such that
\[
\|f_1-f_2\|_X\leq C(M)\omega (\|K(f_1)-K(f_2)\|_Y),
\]
for all $f_1,f_2\in \mathcal{U}_M\cap Q$. Here we set 
$\mathcal{U}_M=\{f\in Z; \|f\|_Z\leq M\}$.
\end{definition}
Here we call $\omega$ the modulus of the conditional stability 
under consideration. 

\iffalse
If
\[
\|f\|_{L^\infty (\Gamma)} +\|g\|_{L^1 (\Gamma)}\leq \varepsilon
\]
and assume $w=u+\mathrm{i}v$ satisfies
\[
\|w\|_{L^\infty(\Omega)}\leq M
\]
\fi

The following harmonic measure will be used to specify stability moduli
for the Cauchy problem.
\begin{definition}[Harmonic measure \cite{Friedman1989}]\label{def-harmonic-measure} 
Let $U\subset\mathbb{C}$ be a simply connected domain with piecewise regular boundary, and $\ell$ be a nonempty open subset of $\partial U$. 
We call $\mu(\zeta)$ the harmonic measure for $U$ and $\ell$, if
%	\begin{align*}
%		\Delta \mu(\zeta) &=0,\quad \zeta\in U\backslash l \label{harmonic-m-1}\\
%		\mu(\zeta) &=0,\quad \zeta\in \partial U\label{harmonic-m-2}\\
%		\mu(\zeta) &=1,\quad \zeta\in l\label{harmonic-m-3}
%	\end{align*}
	\begin{align*}
	\Delta \mu(\zeta) &=0,\quad \zeta\in U \\
	\mu(\zeta) &=0,\quad \zeta\in \partial U\backslash \overline{\ell}\\
	\mu(\zeta) &=1,\quad \zeta\in \overline{\ell}.
\end{align*}
\end{definition}
%
%One has the following properties of the harmonic measure, 
%For the details of which we refer to, e.g., \cite{Friedman1989,Kellogg1953}.
For the details of harmonic measure, we refer to 
\cite{Friedman1989, Kellogg1953} for example.

For a holomophic function $w(z),\, z:= x_1 + \sqrt{-1}x_2 \in\mathbb{C}$
with $x_1, x_2 \in \mathbb{R}$, the following is
known (e.g., \cite{Alessandrini2009}, \cite{Cheng1998}).
\begin{lemma}\label{log-convex-holomorphic}
If $w(z)$ is holomorphic in $\Omega$ and continuous on $\overline{\Omega}$,
and 
\[
|w(z)|\leq \epsilon \quad \forall z\in\Gamma,
\]
\[
|w(z)|\leq M_1\quad \forall z\in\Omega,
\]
where $\epsilon\leq M_1$, then 
\[
|w(z)|\leq M_1\left(\frac{\epsilon}{M_1}\right)^{\tau (z)},
\]
where $\tau$ is the harmonic measure for $\Omega$ and $\Gamma$.
\end{lemma}
\begin{proof}

For $w$, we can construct a subharmonic function $F$:
\[
F(z)=\frac{\ln \left(\frac{|w(z)|}{M_1}\right)}{\ln \left(\frac{\epsilon}{M_1}\right)}.
\]
Then, 
%
%$F$ satisfies
\[
F|_{\Omega}\geq 0,\quad F|_{\Gamma}\geq 1.
\]
The harmonic measure $\tau(z)$ satisfies, $\tau|_{\Gamma}=1,  
\tau|_{\partial\Omega\backslash\Gamma}=0$ and $\tau(z)$ is harmonic in $\Omega$.
Then it holds that
\[
\tau(z)\leq F(z),
\]
which leads to the conclusion that
\[
|w(z)|\leq M_1^{1-\tau(z)}\epsilon^{\tau(z)}.
\]
\end{proof}

The following example illustrates that the estimate for $w(z), z\in\mathbb{C}$ is sharp.
\begin{example}
Suppose that $\Omega=\{z;\, 1\leq |z|\leq R\}$ and  
$\Gamma=\{z;\,|z|=1\}$.  In Lemma 2.1, we consider 
\[
w(z)=\epsilon z^n,\quad n\in\mathbb{N}.
\]
It holds that
\[
\left|\frac{\omega(z)}{M}\right|=\left(\frac{\epsilon}{M}\right)^{\tau(z)},\quad M=|w(z)|_{|z|=R},
\]
which means that the conclusion of Lemma 2.1 is the best possible for 
these $w(z)$.
\end{example}

\begin{theorem}[Conditional stability]\label{thm-conditional-stability-main}
Let $\Omega$ be a simply connected domain in $\mathbb{R}^2$ and let
$\Gamma$ be a non-empty open subset of $\partial\Omega$.
Suppose that $u(x)$ satisfies
\begin{align*}
\Delta u(x)&=0,\quad\quad x\in\Omega,\\
u(x)&=f(x),\quad x\in\Gamma,\\
\partial_\nu u (x)&=g(x),\quad x\in\Gamma.
\end{align*}
If $\|u\|_{C^1(\overline{\Omega})}\leq M$ with arbitrarily 
given constant $M>0$, then we have 
\begin{equation}\label{conditional-stability-1}
|u(x)|\leq C(M,\Omega) \varepsilon^{\tau(x)} \quad \mbox{for $x\in \Omega$},
\end{equation}
where
\[
\varepsilon=\|f\|_{W^{1,\infty}(\Gamma)}+\|g\|_{L^\infty(\Gamma)}
\]
and $\tau(x)$ is the harmonic measure with respect to $\Gamma$ and $\Omega$.
\end{theorem}

\begin{proof}
We define an analytic function $w(z)$ as
\[
w(z)=\frac{\partial u}{\partial \overline{z}}
=\frac{1}{2}\left(\frac{\partial u}{\partial x_1}
+ i\frac{\partial u}{\partial x_2}\right),
\]
which is holomorphic in $\Omega$.

Since $\|u\|_{C^1(\overline{\Omega})}\leq M$, 
one has $\|w\|_{L^\infty(\Omega)}\leq CM$. % for $(x,y)\in\Omega$. %
Since
\[
\|w(z)\|_{L^\infty(\Gamma)}\leq C\varepsilon,
\]
Lemma \ref{log-convex-holomorphic} yields 
\[
|w(z)|\leq C M^{1-\tau(z)}\varepsilon^{\tau(z)}.
\]
For $x\in\Omega$, let $L$ denote a path connecting $x$ and $x_0\in\Gamma$.
Then 
\begin{align*}
|u(x)|=& \left| u(x_0)+\int_{L}\frac{\partial u}{\partial s}(s)\mathrm{d}
\ell_s\right| \\
\leq & |u(x_0)|+\int_{L}|\nabla u|\mathrm{d}\ell_s
\leq  |u(x_0)|+C\int_{L}\varepsilon^{\tau(s)}\mathrm{d}\ell_s,
\quad x_0\in\Gamma
\end{align*}
We claim that there exists a path $L$ from some $x_0\in\Gamma$ to $x$ along which $\tau$ monotonously decreases. Otherwise, the point $x$ must be enclosed by a closed contour on which $\nabla \tau(x)=0$, which implies that $\nabla \tau(x)\equiv 0$ in $\Omega$, leading to a contradict. Therefore, we have
\begin{equation*}
|u(x)|\leq  |u(x_0)|+C\int_{L}\varepsilon^{\tau(x)}\mathrm{d}\ell_s
\leq C\varepsilon^{\tau(x)},
\end{equation*}
which leads to \eqref{conditional-stability-1}.
\end{proof}

In real applications, we often measure the observation errors by 
$L^2(\Omega)$-based norms.  In that case, we can prove

\iffalse
	\begin{figure}[H]
	\centering
		\includegraphics[width=2.5in]{domain3.png}
	%\caption{$\tau$ with various characteristic boundaries. }
	%\label{tau}
	\end{figure}
\fi

\begin{corollary}\label{global-conditional-stability}%[Conditional stability]
Let $\Omega$ be a simply connected domain in $\mathbb{R}^2$ with piecewise 
smooth boundary and let $\Gamma$ be an open subset of $\partial\Omega$.
Suppose that $u(x)$ satisfies
\begin{align*}
\Delta u(x)&=0\quad x\in\Omega,\\
u(x)&=f(x)\quad x\in\Gamma,\\
\partial_{\nu} u(x)&=g(x)\quad x\in\Gamma.
\end{align*}
If $\|u\|_{H^2(\partial\Omega)}\leq M$, then  
\[
|u(x)|\leq C(M,\Omega,\Gamma) \widetilde{\varepsilon}^{\tau(x)},
\quad x\in \Omega,
\]
where $\widetilde{\varepsilon}=\varepsilon^{\tau_0}$, $\tau_0>0$, 
$\varepsilon=\|f\|_{H^1(\Gamma)}+\|g\|_{L^2 (\Gamma)}$ and
$\tau(x)$ denotes the harmonic measure with respect to $\Gamma$ and $\Omega$.

\end{corollary}

\begin{proof}

Since $\|u\|_{H^1(\Gamma)}\leq \varepsilon$ and $\|u\|_{H^2(\Gamma)}\leq M$, by Sobolev interpolation (e.g., \cite{Isakov2006}), we obtain
\[
\|u\|_{H^{s}(\Gamma)}\leq C(\Gamma)\|u\|_{H^1(\Gamma)}^{\theta}\|u\|^{1-\theta}_{H^2(\Gamma)}= C(\Gamma,M)\varepsilon^\theta,
\]
with $s=\theta+2(1-\theta)> 3/2$ when $0<\theta<1/2$.

The Sobolev embedding (e.g., \cite{Adams2003}) implies
\[
\|u\|_{W^{1,\infty}(\Gamma)} \leq C\|u\|_{H^s(\Gamma)}  \leq C(M,\Gamma)\varepsilon^{\tau_0}=:C(M,\Gamma)\widetilde{\varepsilon},\quad 0<\tau_0<\frac{1}{2}.
\]
Then, by applying Theorem \ref{thm-conditional-stability-main} we have
\[
|u(x)|\leq C(\Gamma,\Omega,M)\widetilde{\varepsilon}^{\tau(x)},
\quad x\in\Omega.
\]
\end{proof}

We remark that even if $u(x)$ is less regular in $\Omega$, 
one can also have similar estimate in a subset of $\Omega$ whose regularity can be ensured due to the interior regularity (e.g., \cite{Gilbarg1983}) 
of elliptic equations, which is standard and is not shown here.

\section{Numerical method}

The solution $u(x)$ to a Laplace equation in the domain $\Omega$ can be represented by Green's function as
\[
u(x)=\int_{\partial \Omega}\frac{\partial G(x,\xi)}{\partial \nu}b(\xi)\mathrm{d}s_\xi,
\]
where $b(x)=u(x)|_{\partial\Omega}$ is the boundary value function. 
The Green function $G$ satisfies
\[
\Delta_x G(x,y) = \delta(x-y),\quad x,y \in \Omega, \quad
G(x,y) = 0, \quad x\in \partial\OOO,\, y\in \Omega.
\]
By $H(x,y)=\partial_\nu G(x,y)$ we denote the Poisson kernel.
Then 
\begin{align*}
&\int_{\partial\Omega}H(x,\xi) b(\xi)\mathrm{d}s_\xi=f(x),\quad x\in\Gamma ,\\
&\int_{\partial\Omega}\nu_\Gamma \cdot \nabla H(x,\xi)b(\xi)\mathrm{d}s_\xi=g(x),\quad x\in\Gamma .
\end{align*}
Then in solving the Cauchy problem, $u|_{\Omega}$ can be determined once the boundary value function $u|_{\partial\Omega}=b(x)$ is recovered from the integral functions.

%	\begin{figure}[H]
%	\centering
%		\includegraphics[width=2.5in]{domain4.png}
%	\caption{Illustration of Runge's approximation.}
	%\label{tau}
%	\end{figure}

For technical reasons, in computation we will reconstruct the harmonic function on a slightly larger domain than $\Omega$ by Runge's approximation, in order to meet the regularity requirements in Theorem \ref{thm-conditional-stability-main}.
We take $\widetilde{\Omega}$ such that $\overline{\Omega} \subset 
\widetilde{\Omega}$. %, $\partial\Omega \cap \partial\Omega'=\Gamma$.
By Runge's approximation \cite{Lax1956}, one can approximate the harmonic function $u(x)$ on $\Omega$ by a harmonic function $\www{u}(x)$ in $\www{\Omega}$, which will be illustrated in the following section. Let $\www{G}(x,y)$ be the 
Green function corresponding to $\www{\Omega}$, and denote $\www{H}(x,y)
=\partial_{\www{\nu}} \www{G}(x,y)$ where $\www{\nu}$ is the outer normal 
vector to $\partial\www{\Omega}$. Correspondingly, we denote
\begin{align*}
\www{f}(x):=\www{u}|_{\Gamma}=&\int_{\partial\www{\Omega}}\www{H}(x,\xi) 
\www{b}(\xi)\mathrm{d}s(\xi),\quad x\in\Gamma,\\
\www{g}(x):=\partial_{\nu_\Gamma}\tilde{u}|_{\Gamma}=&\int_{\partial\www{\Omega}}\nu_\Gamma \cdot\nabla \www{H}(x,\xi)\tilde{b}(\xi)\mathrm{d}s(\xi),
\quad x\in\Gamma.
\end{align*}
Then we solve $\widetilde{b}$ from the measurements $f^\varepsilon, g^\varepsilon$ 
with noises, and reconstruct $\www{u}(x)|_{\Omega}$. %$u|_{\Omega}$ can be determined once the boundary value function $f_N(x)$ is recovered from the measurements.

Due to the ill-posedness of the problem, the Tikhonov regularization is introduced to weaken the instability induced by the observation error. According to the conditional stability discussed in Section 2, the regularized cost functional is defined as
\[
J(\www{b}):=\|\www{f}(\www{b})-f^\varepsilon \|^2_{L^2(\Gamma)}
+ \|\www{g}(\www{b})-g^\varepsilon \|^2_{L^2(\Gamma)}
+ \alpha\|\www{b}\|^2_{L^2(\partial\www{\Omega})}.
\]
Then we obtain the solution which minimizes the cost functional.

To discretize the problem, suppose that $\www{b}(x)$ can be approximated by
\[
\www{b}_n(x)=\sum_{i=1}^n b_i\varphi_i(x),
\]
where $\varphi_i(x)$ are basis functions defined on $\partial\www{\Omega}$, 
and $b_i$ are the corresponding components.
The test space $V_n=span\{\varphi_j\}_{j=1}^n$ is chosen such that 
$\cup^\infty_{j=n+1}V_j$ is dense in $L^{2}(\partial\www{\Omega})$. Let
\[
w_i=\int_{\partial\www{\Omega}}\frac{\partial \www{G}}{\partial \nu}
\varphi_i\mathrm{d}S,\quad i=1,...,n
\]
satisfy
\[
\left\{
\begin{array}{l}
\Delta w_i =0\\
\left. w_i\right|_{\partial\www{\Omega}}=\varphi_i.
\end{array}
\right. 
\]
Then the harmonic function in $\tilde{\Omega}$ can be approximated by
\[
\tilde{u}_n (x) =\sum_{i=1}^n b_i w_i(x).
\]

In particular, let
\begin{equation*}
w_i=\int_{\gamma_i}\frac{\partial G}{\partial \nu}\mathrm{d}s, \quad 
\bigcup_i \gamma_i=\partial\www{\Omega},\quad \gamma_i\bigcap_{i \neq j}\gamma_j=\emptyset.
\end{equation*}
Then the solution can be expressed by
\begin{equation*}
\www{u}_n(x)= \sum_{i=1}^n b_i \int_{\gamma_i}\frac{\partial \www{G}}{\partial \nu}(x,s)\mathrm{d}s=:  \sum_{i=1}^n b_i w_i(x),
\end{equation*}
Then $w_i$ satisfies
\begin{equation}\label{num-base-solution}
\left\{ \begin{array}{l}
\Delta w_i (x)=0, \quad x\in\,\www{\Omega}, \\
\left. w_i\right|_{\partial\Omega}=\chi(\Gamma_i),\quad x\in\, 
\partial\www{\Omega}.
\end{array}\right. 
\end{equation}

Notice that the base solution $w_i$ involves the singular integral and
one can approximate it by solving \eqref{num-base-solution} numerically, which are denoted by $w^h_i(x), i=1,\cdots, n$. Here we use $h$ to mark the discrete 
precision in calculating $w_i(x)$. In this work, we choose the 
finite difference method (FDM) to numerically compute $w^h_i(x)$ with 
the grid length $h$. The 2nd order center difference discretization will be adopted.

Let $\mathcal{F}_{n,h}(\www{\Omega}) = span_{1\le i\le n} w_i^h(x)$ be 
the space spanned by the base solutions. Then our problem is to find
\begin{equation}\label{mini-solution}
u^*=\mathop { \arg\min}_ {\www{u}^h_n\in \mathcal{F}_{n,h}(\www{\Omega})}
\left\{
\|\www{f}^h_n-f^\delta \|^2_{H^1(\Gamma)}+\|\www{g}^h_n-g^\delta\|^2_{L^2(\Gamma)}+\alpha\|\www{u}^h_n\|^2_{H^2(\partial \Omega)}
\right\},
\end{equation}
or alternatively,
\begin{equation}\label{mini-solution}
\www{b}_n^*=\mathop { \arg\min}_ {\tilde{b}_n\in V_n}\left\{
\|\www{f}^h_n(\www{b}_n)-f^\delta \|^2_{H^1(\Gamma)}+\|\www{g}^h_n(\www{b}_n)-g^\delta\|^2_{L^2(\Gamma)}+\alpha\|\mathcal{H}\www{b}_n\|^2
_{H^2( \partial\Omega)}
\right\},
\end{equation}
where $\www{f}^h_n=\www{u}^h_n|_{\Gamma}$, and $\www{g}^h_n=\nu_\Gamma\cdot\nabla_h \www{u}^h_n|_{\Gamma}$, $\nabla_h$ is the numerical gradient. 

Then, $\mathcal{H}\www{b}_n=\int_{\partial\www{\Omega}}\www{H}(x,\xi)\www{b}_n
(\xi)\mathrm{d}s$.
According to the a priori choice strategy of the regularization parameter 
(\cite{Cheng2000-1}), $\alpha$ is taken as $\alpha\sim\delta^2$.

We assume that the measurements are taken at $x_j\in \Gamma, j=1,\cdots, m$ and let
%Denote the discrete points on the measurement boundary $\Gamma$ as $x_j$, we have
\[
A_{ji}:=\www{f}^h_i(x_j)=w^h_i(x_j),\quad B_{ji}:=\www{g}^h_i(x_j)=\partial_{\nu}w_i^h(x_j),\quad C_k:=\|w_k^h(x)\|_{H^2(\partial\Omega)}.
\]
Then we reach the fully discrete form of the regularization cost functional:
\[
\widehat{J}(\mbf{b}):=\sum_{j=1}^m\left(  \sum_{i=1}^n A_{ij}b_i-f^\varepsilon_j\right)\sigma_j+
\sum_{j=1}^m\left(  \sum_{i=1}^n B_{ij}b_i-g^\varepsilon_j\right) \sigma_j+\alpha^2\sum_{k=1}^N(C_kb_k)^2,
\]
where $\sigma_j$ denotes the $j$th curve length element on $\Gamma$ and $\mbf{b}=(b_1,\cdots,b_n)\in\mathbb{R}^n$. The minimization is a standard linear algebra problem.

\iffalse

��ֵGreen���������ж���Ӧ�á���������Cauchy���⣬
\begin{equation}
\left\{
\begin{array}{l}
\Delta u=0 \quad \text{in} \Omega\\
u\left.\right|_{\Gamma\subset \partial \Omega}=f,\quad \frac{\pl u}{\pl n}\left.\right|_\Gamma=g
\end{array}
\right.
\end{equation}
�ع� $u$���� $u\left. \right|_{\partial\Omega /\Gamma}=f_1$�����ع������ȼ�������$min || \frac{\partial u}{\partial n}(f_1)-g||_{\Gamma}$.
��ɢ�����£�������
\begin{equation}
 \sum_j f_1(x_j)\frac{\partial \omega_j}{\partial n}(x_i)+\sum_k f(x_k)  \frac{\partial \omega_k}{\partial n}(x_i)=g(x_i),
\end{equation}
\begin{equation}
 \sum_j f_1(x_j)\frac{\partial \omega_j}{\partial n}(x_i)=g(x_i) -\sum_k f(x_k) \frac{\partial \omega_k}{\partial n}(x_i),
\end{equation}
\begin{equation}
\sum_j K_{ij}a_j=b_i,\quad K_{ij}=\frac{\partial \omega_j}{\partial n}(x_i),\quad a_j =f_1(x_j),\quad b_i=g(x_i) -\sum_k f(x_k) \frac{\partial \omega_k}{\partial n}(x_i)
\end{equation}
���� $\omega_j$Ϊ
\[
\left\{
\begin{array}{l} \Delta \omega_j=0 \\ \omega_j \left.\right|_{\partial\Omega}=\chi_{\Gamma_j}(x)\end{array}\right. .
\]

\fi

\section{Error analysis}
In the following, we assume that the exact solution $u$ has enough regularity 
in $\Omega$.
Otherwise one can utilize the interior regularity and pay attention to the reconstruction on any subset whose closure is contained in $\Omega$. The main result on pointwise evaluation of the reconstructed solution is as follows.

\begin{theorem}[Evaluation on $\Omega$]\label{thm-error-point}
%Let $\tilde{u}_n$ be the minimizer to \eqref{mini-solution},
Suppose that $\overline{\Omega}\subset \www{\Omega}$ and 
$u_0$ is harmonic in $\Omega$, and 
$\|u_0\|_{H^2(\partial\Omega)}\leq M$. 
Denote $f_0=u_0|_{\Gamma}$ and $g_0=\partial_\nu u_0|_{\Gamma}$.
Let available data $f^\varepsilon$, $g^\varepsilon$ satisfy $\|f^\varepsilon-f_0\|_{H^1(\Gamma)}+\|g^\varepsilon-g_0\|_{L^2 (\Gamma)}\leq \varepsilon $.
Following the scheme presented in Section 3, by $u^*$ we denote 
the minimizer of
\begin{equation}\label{mini-solution}
u^*=\mathop { \arg\min}_ {\widetilde{u}^h_n\in \mathcal{F}^h_n(\widetilde{\Omega})}\left\{
\|\www{f}^h_n-f^\varepsilon \|^2_{H^1(\Gamma)}+\|\www{g}^h_n
-g^\varepsilon\|^2_{L^2(\Gamma)}+\alpha\|\www{u}^h_n\|^2_{H^2(\partial\Omega)}
\right\}.
\end{equation}
%where $\alpha$ is taken as $\alpha\sim\varepsilon^2+\delta(n)^2+h^2$.
Then, we have the estimate for the Cauchy problem:
\begin{equation}
|u^*(x) - u_0(x)|\leq C(M,\Omega,\Gamma)\varepsilon^{\tau(x)}, \quad
x\in \Omega,
\end{equation}
provided that $\alpha\sim\varepsilon^2$, $n$ are sufficiently large and $h$ is
sufficiently small. 
Here $\tau(x)$ is the harmonic measure with characteristic boundary $\Gamma$.
\end{theorem}

%To prove the theorem, we will summarize the involved discrete errors, and give the estimate on the Cauchy data and the main result on the reconstruction error in the end.  

\begin{lemma}\label{estimate-num-Green}
Suppose that $\www{u}^h_n$ is the FDM approximations of a harmonic function 
$\www{u}$ in $\www{\Omega}$ with the scheme given in Section 3.
We set $\www{f}=\www{u}|_{\Gamma}$, $\www{g}=\partial_\nu \www{u}|_{\Gamma}$, 
$\www{f}^h_n=\www{u}^h_n|_{\Gamma}$, $\www{g}^h_n=\nu\cdot\nabla_h \www{u}^h_n|_{\Gamma}$ with $\nabla_h$ which is the gradient 
approximated by the 1st order difference.
Then 
\[
\|\www{f}_n^h-\www{f}\|_{H^1( \Gamma) }\leq C_1\delta(n)+C_2 h,
\]
\[
\|\www{g}_n^h-\www{g}\|_{L^2( \Gamma)}\leq C_3 \delta(n)+C_4h,
\]
where $\delta (n)\rightarrow 0$ as $n\rightarrow \infty$ and $C_1, C_2, C_3, 
C_4$ are constants depending on $\Gamma, \www{\Omega}$ and $\www{u}$.
\end{lemma}
\begin{proof}
Define $\www{u}_n(x)= \int_{\partial\www{\Omega}}\partial_\nu \www{G}(x,s)
\www{b}_n(s)\mathrm{d}s$ for $x\in\overline{\Omega}$.
Then, 
\begin{equation*}
|\www{u}(x) - \www{u}_n^h(x)|\leq  \left| \www{u}(x) - \www{u}_n(x)\right|
+ |\www{u}_n(x) - \www{u}_n^h(x) |.
\end{equation*}
According to the interior regularity of the Laplace equation, 
since the above $\www{u}_n$ and $\www{u}
= \int_{\partial\www{\Omega}}\partial_\nu \www{G}(x,s)\www{b}(s)\mathrm{d}s$ 
are both harmonic in $\www{\Omega}$,  we see 
\begin{equation*}
\left| \www{u}(x)-\www{u}_n(x)\right|\leq \|\www{u}-\www{u}_n\|
_{C(\overline{\Omega})}
\leq C(\Omega,\www{\Omega}) \|\www{b}-\www{b}_n\|_{L^2(\partial\www{\Omega})}.
\end{equation*}
Due to the density of the space of test functions, we have 
$|\www{u}(x)-\www{u}_n(x)|\leq C\delta(n)$.
For the second part, a standard error estimate for the second order central difference method (e.g., \cite{Larsson2003}) implies 
\[
|\www{u}_n(x)-\www{u}_n^h(x) |\leq C\|\www{u}_n\|_{C^4(\overline{\Omega})}h^2
\leq C(\Omega,\www{\Omega})\|\www{u}_n\|_{L^2(\www{\Omega})}h^2
\leq C(\Omega,\www{\Omega},\www{u}_n)h^2
\]
for $x\in\Omega$. 
Combining the precision of the 1st order difference for the gradient, the estimate for $\www{f}^h_n$ and $\www{g}^h_n$ can be obtained.

\end{proof}

The above estimate means that the error can be decomposed by the boundary discrete part and the FDM discrete part, and converges as $n\rightarrow\infty$ and $h\rightarrow 0$.

\begin{lemma}\label{residual-estimate}%[Residual on $\Gamma$]
Under the assumption of Theorem \ref{thm-error-point},
by the scheme presented in Section 3, let $u^*$ be constructed as the
minimizer: 
\begin{equation}\label{mini-solution}
u^*=\mathop { \arg\min}_ {\www{u}^h_n\in \mathcal{F}^h_n(\www{\Omega})}
\left\{ \|\www{f}^h_n-f^\varepsilon \|^2_{H^1(\Gamma)}
+ \|\www{g}^h_n-g^\varepsilon\|^2_{L^2(\Gamma)}
+ \alpha\|\www{u}^h_n\|^2_{H^2(\partial\Omega)}
\right\},
\end{equation}
where $\alpha$ is taken as $\alpha\sim\varepsilon^2+\delta(n)^2+h^2$.
Then, we have
\[
\| f^*-f_0 \|_{H^1(\Gamma)}+\|g^*-g_0\|_{L^2(\Gamma)}
\leq C(M,\Omega,\www{\Omega})(\varepsilon+\delta(n)+h), 
\]
where $f^*=u^*|_{\Gamma}$ and $g^*=\partial_\nu u^*|_{\Gamma}$.
\end{lemma}

\begin{proof}
First, by Runge's approximation (e.g., \cite{Salo2018}, \cite{Lax1956}), 
there exists a harmonic function $\widetilde{u}$ in $\widetilde{\Omega}$
such that 
\[
\|u_0-\www{u}\|_{H^2(\Omega)}\leq \varepsilon.
\]
Then, 
\begin{equation}\label{bound-u-tilde}
\|\www{u}\|_{H^2(\Omega)}\leq CM,
\end{equation}
and by the trace theorem (\cite{Adams2003}),
\begin{equation}\label{err-f-tilde}
 \|\www{f}_0-f_0\|_{H^1(\Gamma)}+\|\www{g}_0- g_0\|_{L^2(\Gamma)}
\leq \varepsilon,
\end{equation}
where $\www{f}_0=\www{u}|_{\Gamma}$ and $\www{g}_0=\partial_\nu \www{u}|
_{\Gamma}$.

The definition of the minimizer yields
\begin{align}
& \| f^*-f^\varepsilon \|^2_{H^1(\Gamma)}+\|g^*-g^\varepsilon\|^2_{L^2(\Gamma)}+\alpha\|u^*\|^2_{H^2(\partial\Omega)}\notag\\
\leq &  \|\www{f}^h_{0,n}-f^\varepsilon \|^2_{H^1(\Gamma)}+\|\www{g}^h_{0,n}-g^\varepsilon\|^2_{L^2(\Gamma)}+\alpha\|\www{u}_{0,n}^h\|^2_{H^2(\partial\Omega)},\label{minimizer-inequality}
\end{align}
where $\www{f}^h_{0,n}=\www{u}^h_{0,n}|_{\Gamma}$ and so is $\www{g}^h_{0,n}
=\partial_\nu\www{u}^h_{0,n}|_{\Gamma}$.
Therefore,
\begin{equation*}
\alpha \|u^*\|^2_{H^2(\partial\Omega)}\leq
 \|\www{f}_{0,n}^h-f^\varepsilon \|^2_{H^1(\Gamma)}+\|\www{g}_{0,n}^h-g^\varepsilon\|^2_{L^2(\Gamma)}+\alpha \|\www{u}_{0,n}^h\|^2_{H^2(\partial\Omega)}.
\end{equation*}
We can estimate the first term on the right hand side as
\[
\|\www{f}_{0,n}^h-f^\varepsilon \|_{H^1(\Gamma)} \leq \|\www{f}_{0,n}^h
-\www{f}_0 \|_{H^1(\Gamma)}+\|\www{f}_0-f_0 \|_{H^1(\Gamma)}+\|f_0- f^\varepsilon \|_{H^1(\Gamma)}.
\]
The second term $\|\www{g}_{0,n}^h-g^\varepsilon\|^2_{L^2(\Gamma)}$ is dealt with similarly.
Based on Lemma \ref{estimate-num-Green},
% and the estimate of $\tilde{u}$ \eqref{bound-u-tilde},
%Lemma \ref{runge-approx}
we obtain
\[
\|\www{f}_{0,n}^h-\www{f}_0 \|_{H^1(\Gamma)}+\|\www{g}_{0,n}^h-\www{g}_0 \|_{L^2(\Gamma)} \leq  C_1\delta(n)+C_2h,
\]
where $C_1,C_2$ depend on $\Gamma,\www{\Omega},M$. %
By \eqref{err-f-tilde} and the assumption that $\|f^\varepsilon-f_0\|_{H^1(\Gamma)}+\|g^\varepsilon-g_0\|_{L^2 (\Gamma)}\leq \varepsilon $, we
 see 
\[
 \|\www{f}_{0,n}^h-f^\varepsilon \|_{H^1(\Gamma)}+\|\www{g}_{0,n}^h-g^\varepsilon\|_{L^2(\Gamma)}\leq C_3\delta(n)+C_4 h+C_5 \varepsilon,
\]
where the constants $C_3, C_4, C_5$ depend on $M,\Gamma,\www{\Omega}$. 
Meanwhile,
\[
\|\www{u}^h_{0,n}\|_{H^2(\partial\Omega)}\leq \|\www{u}^h_{0,h}-\www{u}_0\|_{H^2(\partial\Omega)}+\|  \www{u}_0-u_0 \|_{H^2(\partial\Omega)}+\|u_0\|_{H^2(\partial\Omega)}\leq C_6M.
\]
%assuming that $M$ is order dominant over $h$ and $\delta(n)$.
%
Consequently,
\[
\|u^*\|^2_{H^2(\partial\Omega)}
\leq C_7(\Gamma,\www{\Omega},M)\frac{\delta^2(n)+ h^2+\varepsilon^2}{\alpha}
+C_6(\Gamma,\www{\Omega})M^2.
\]
%{\color{blue}Should proof that $\tilde{u}^h_n\leq M$}
%
With the choice of $\alpha$, one reaches 
\begin{equation}\label{bound-tildeU}
\|u^*\|_{H^2(\partial\Omega)}\leq C'(\Gamma, \www{\Omega},M).
\end{equation}

For the residual part, in view of \eqref{minimizer-inequality} and the above estimate we have
\begin{align*}
& \| f^*-f^\varepsilon \|^2_{H^1(\Gamma)}+\| g^*-g^\varepsilon\|^2_{L^2(\Gamma)}\\
\leq & \|\www{f}_{0,n}^h-f^\varepsilon \|^2_{H^1(\Gamma)}+\|\www{g}_{0,n}^h-g^\varepsilon\|^2_{L^2(\Gamma)}+\alpha\|\www{u}_{0,n}^h\|^2_{H^2(\Omega)}\\
\leq &  C_7\left(\delta^2(n)+ h^2+\varepsilon^2\right)+\alpha C_6M^2.
\end{align*}
Therefore, with the choice of $\alpha$, we have
\begin{align*}
& \| f^*-f^\varepsilon \|_{H^1(\Gamma)}+\| g^*-g^\varepsilon\|_{L^2(\Gamma)}\\
\leq & C''(\Gamma,\www{\Omega},M)(\varepsilon+\delta(n)+h).
\end{align*}

Finally, by combining the boundness of $u^*$ and the estimate on the space
of test functions, we have
\begin{align*}
\|f^*-f_0\|_{H^1(\Gamma)}+ \|g^*-g_0\|_{L^2(\Gamma)} & \leq \|f^*-f^\varepsilon\|_{H^1(\Gamma)}+\|f^\varepsilon-f_0\|_{H^1(\Gamma)}\\
& +\|g^*-g^\varepsilon\|_{L^2(\Gamma)}+\|g^\varepsilon-g_0\|_{L^2(\Gamma)}\\
& \leq C(\Gamma,\www{\Omega},M)(\varepsilon+\delta(n)+h),
\end{align*}
which is the conclusion of the lemma.
%where $C_1,C_2$ and $C_3$ depend on $\Gamma,\Omega$ and $M$.
\end{proof}

%----------------------------------------------

\begin{lemma}\label{boundness-difference}
Denote $\mbf{b}^*=(b^*_1,\cdots, b^*_n)$ as the vector corresponds to $u^*=\sum_{i=1}^nb^*_iw^h_i(x)$ where $u^*$ is the minimizer in Lemma \ref{residual-estimate}, and $u^*_n=\sum_{i=1}^nb^*_iw_i(x)$, where $w_i(x)$ are
the base functions defined in Section 3 and $w_i^h(x)$ the
numerical approximations. Under the assumption of 
Lemma \ref{residual-estimate}, we have 
\[
\| u^*_n-u_0 \|_{C^1(\overline{\Omega})}\leq C(\Gamma,\www{\Omega},M),
\]
provided that $n$ is sufficiently large.
\end{lemma}
\begin{proof}
First,
\[
\|u_n^*\|_{H^2(\partial\Omega)}\leq \|u_n^*-u^*\|_{H^2(\partial\Omega)}+\|u^*\|_{H^2(\partial\Omega)}.
\]
By means of the boundness \eqref{bound-tildeU} and the convergence of the FDM, 
we see 
\[
\|u_n^*\|_{H^2(\partial\Omega)}\leq CM.
\]
The assumption of $u_0$ and the Sobolev embedding (\cite{Adams2003}) yield
\[
\| u^*_n-u_0 \|_{C^1(\overline{\Omega})}\leq   \| u^*_n-u_0 \|_{H^2(\partial\Omega)}\leq     \|u^*_n\|_{H^2(\partial\Omega)}+ \|u_0\|_{H^2(\partial\Omega)}\leq C(\Gamma,\www{\Omega},M).
\]
\end{proof}

After having the estimate on $\Gamma$ and the boundness on $\Omega$, we can further have the reconstruction error on $\Omega$.

%---------------------------------------------------------------------------------------

\begin{proof}[Proof of Theorem 4.1]
We note that $u^*_n-\widetilde{u}$ is harmonic and is bounded due to Lemma \ref{boundness-difference}. By Lemma \ref{residual-estimate}
we can obtain
\begin{equation*}
	\|f^*_n-f_0\|_{H^1(\Gamma)}+\|g^*_n-g_0\|_{H^1(\Gamma)}
\leq  C(\Gamma,\www{\Omega},M)(\varepsilon+h+\delta(n)),
\end{equation*}
where we have used
\[
\|f^*_n-f_0\|_{H^1(\Gamma)}\leq \|f^*_n-f^*\|_{H^1(\Gamma)}+\|f^*-f_0\|_{H^1(\Gamma)}\leq C_1(\Gamma,\www{\Omega},M)h+C_2(\Gamma,\www{\Omega},M)(\varepsilon+h+\delta(n)).
\]
Then we apply the conditional stability result in Theorem \ref{thm-conditional-stability-main} to reach 
\[
|u^*_n(x) -u_0(x)\vert \leq C_3(\Gamma,\www{\Omega},M)(\varepsilon+h+\delta(n))^{\tau(x)}.
\]
Finally, for the approximation error by the FDM for $u^*_n$, we have
\begin{align*}
&|u^*(x)-u_0(x)| \leq |u^*(x)-u^*_n(x)|+|u^*_n(x)-u_0(x)| \\
&\leq C_4h+C_3(\Gamma,\www{\Omega},M)(\varepsilon+h+\delta(n))^{\tau(x)}\leq C(\Gamma,\www{\Omega},M)\varepsilon^{\tau(x)}
\end{align*}
for $x\in\Omega$, provided that $\delta(n)$ and $h$ are sufficiently small.
\end{proof}

\section{Numerical examples}
\paragraph{Numerical example}
We have applied the above numerical methods to various cases. We will demonstrate the performances including the reconstruction evaluations.

We consider the domain $\Omega=(0,1)\times (0,1)$ and the measurement boundary $\Gamma=\{(x_1,0);\, 0<x_1<1\}$. The exact solution is selected as 
$u(x_1,x_2)=e^{4x_1} \cos 4(x_2+0.2)$.
The result with noise level 1\% is displayed in Figure \ref{eg1}, where we adopt the method in Section 3. We choose $\varphi_i(x_1,x_2),i=1,\cdots,n$ as linear interpolation bases along the boundary of the FEM grids. We discussed the case $n=264$ and $h=1/64$ in the FDM calculation. 
The measurement points coincide with the FDM grids.

	\begin{figure}[H]
	\centering
	\subfigure[Exact solution]
		{
		\includegraphics[width=1.7in]{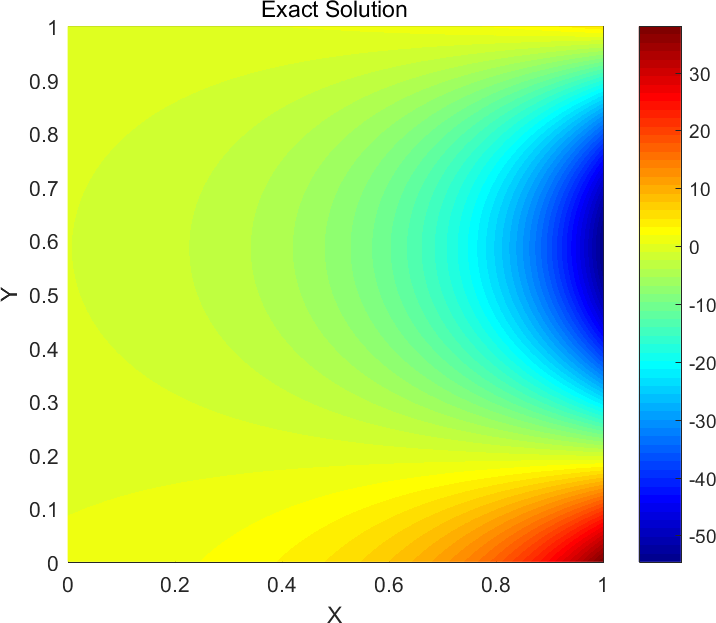}
		}
	\subfigure[Reconstruction]
		{
		\includegraphics[width=1.7in]{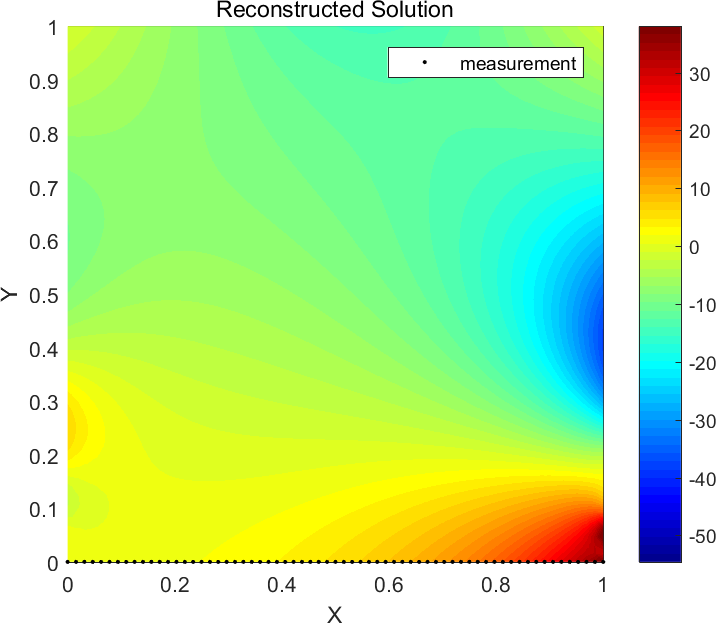}
		}
	\subfigure[Absolute error]
		{
		\includegraphics[width=1.7in]{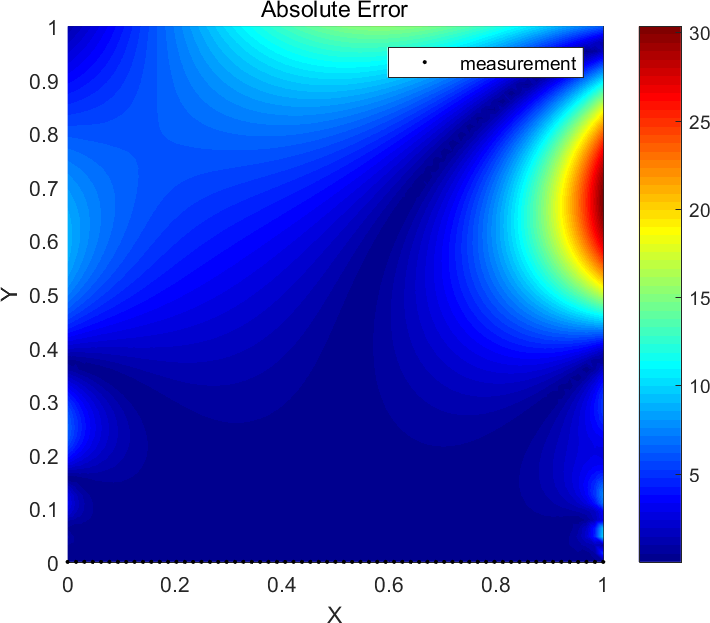}
		}			
	\caption{Reconstruction with one measurement boundary.}
	\label{eg1}
	\end{figure}

For the error in the reconstruction domain, it is expected that the error in 
$\Omega$ will be amplified significantly when getting far from the measurement boundary, due to the H{\"o}lder-type stability index indicated in Section 2. Since the estimate is sharp, even though the error is small somewhere far from the bottom side, the result is not reliable.

	\begin{figure}[H]
	\centering
	\subfigure[Exact solution]
		{
		\includegraphics[width=1.7in]{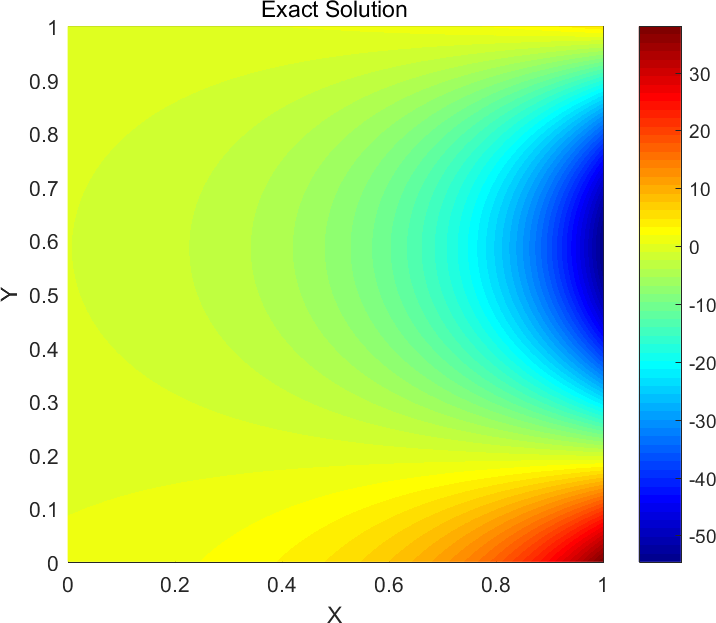}
		}
	\subfigure[Reconstruction]
		{
		\includegraphics[width=1.7in]{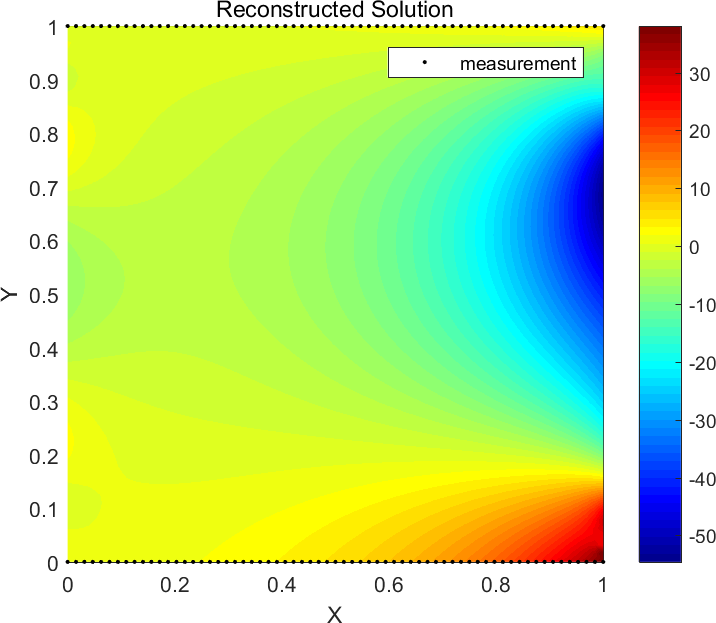}
		}
	\subfigure[Absolute error]
		{
		\includegraphics[width=1.7in]{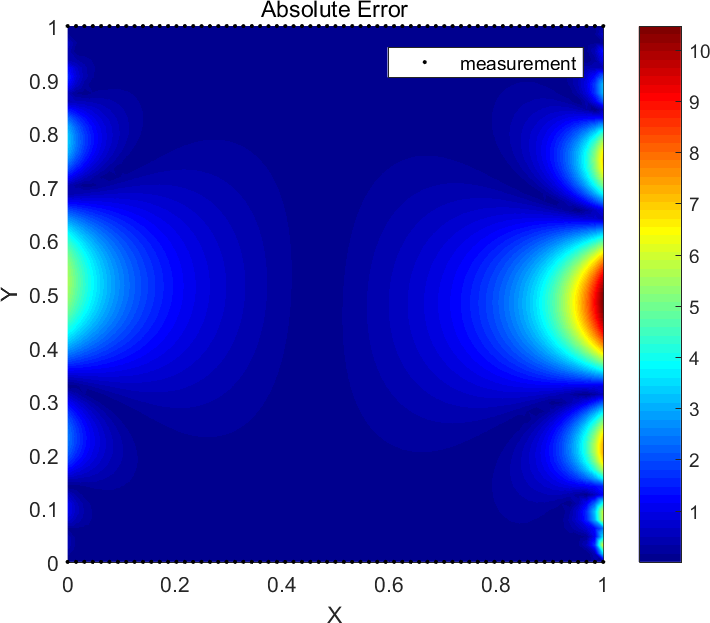}
		}				
	\caption{Reconstruction with two measurement boundaries. }
	\label{eg2}
	\end{figure}

 Figure \ref{eg2} gives the case with an additional measurement boundary on the upper side. The result near the upper boundary is greatly improved comparing with the single measurement case. On the lateral sides, there is no measurement and the result there is not reliable, although the error level is not large in some places. This will be further illustrated in the following.

\paragraph{Indicate function and reliable reconstruction domain}
The error estimate implies that the reconstruction error will no longer be improved by increasing the discrete accuracy once the observation error becomes dominant. Meanwhile, even if the observation error is small, the error far from the measurement area may be enlarged significantly, that is, the reconstruction there is not reliable. These phenomena are caused by the ill-posedness of the problem. Now that the reconstruction accuracy on the whole reconstruction domain are not ensured, one hopes to know where the reconstruction error has acceptable convergence rate with observation noises and discrete errors in real applications. The reliable domain can be determined by the pointwise error estimate in Theorem \ref{thm-error-point}. Since the error growth rate depends on $\tau(x_1,x_2)$,
it further depends on the shape of the reconstruction domain. Figure \ref{tau} 
indicates profiles of the harmonic measure corresponding to the present case with a rectangle computational area. The area bounded by the black curve and the measurement boundary corresponds to $\tau(x_1,x_2)>0.5$, which may be regarded as confidence area in practice. This is consistent with the error distributions in the examples (see Figures \ref{eg1} and \ref{eg2}). The area with convergence rate higher than 0.5 is enlarged significantly by adding measurement boundaries.%Figure \ref{harmonic-measure-illustration}.

	\begin{figure}[H]
	\centering
	\subfigure[One observation side.]
		{
		\includegraphics[width=2.2in]{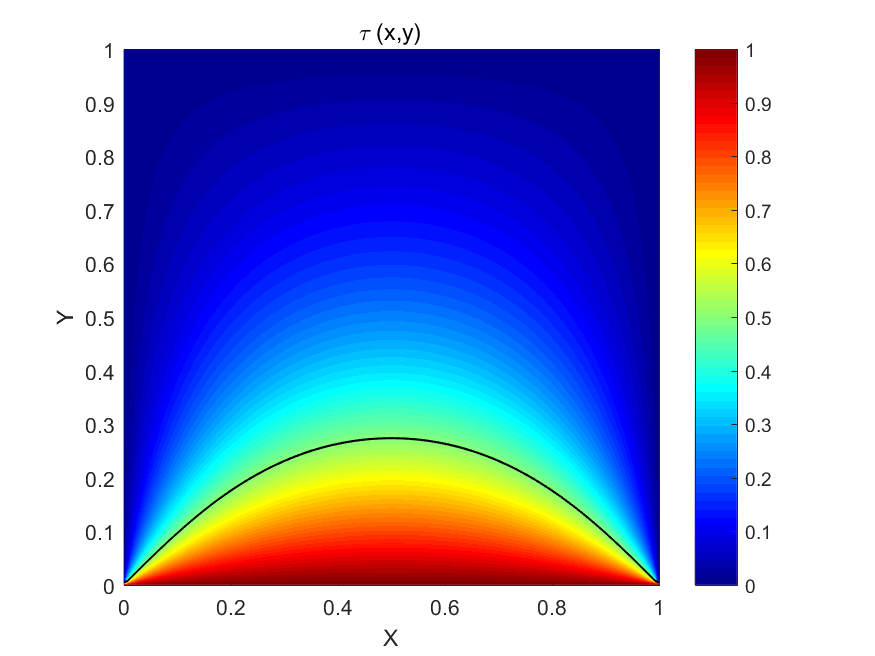}
		}
	\subfigure[Two observation sides.]
		{
		\includegraphics[width=2.2in]{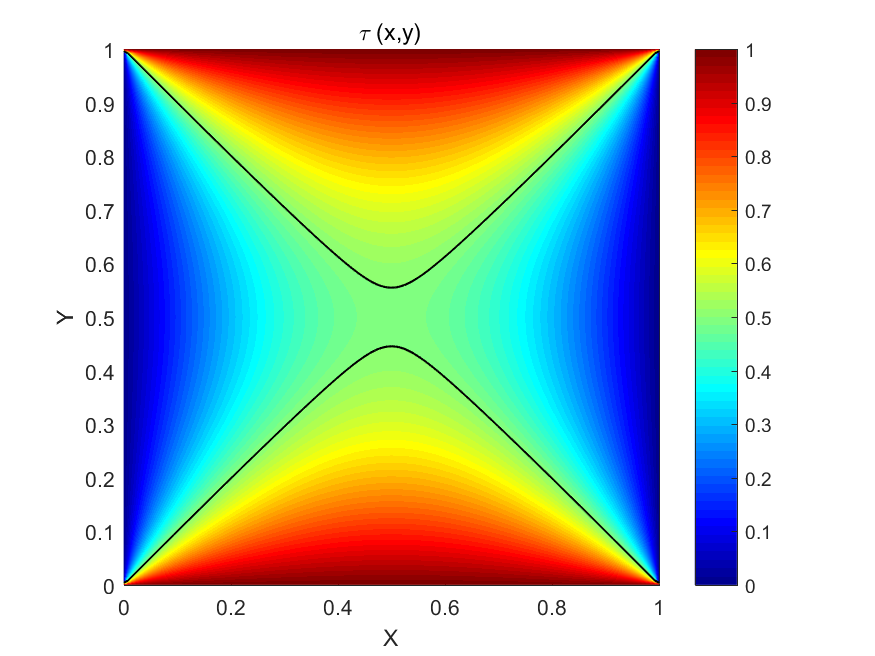}
		}
	\subfigure[Two observation sides.]
		{
		\includegraphics[width=2.2in]{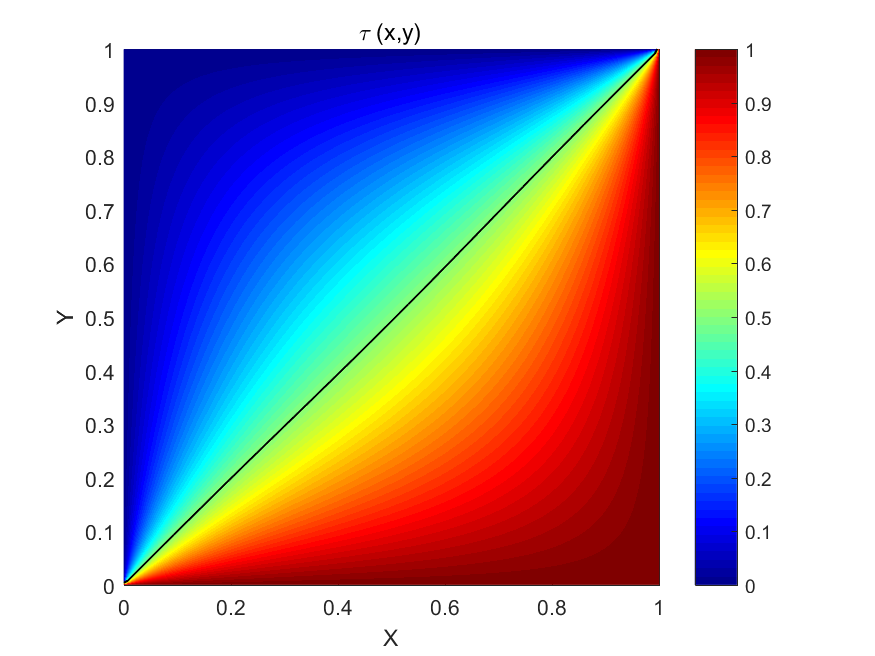}
		}			
	\subfigure[Three observation sides.]
		{
		\includegraphics[width=2.2in]{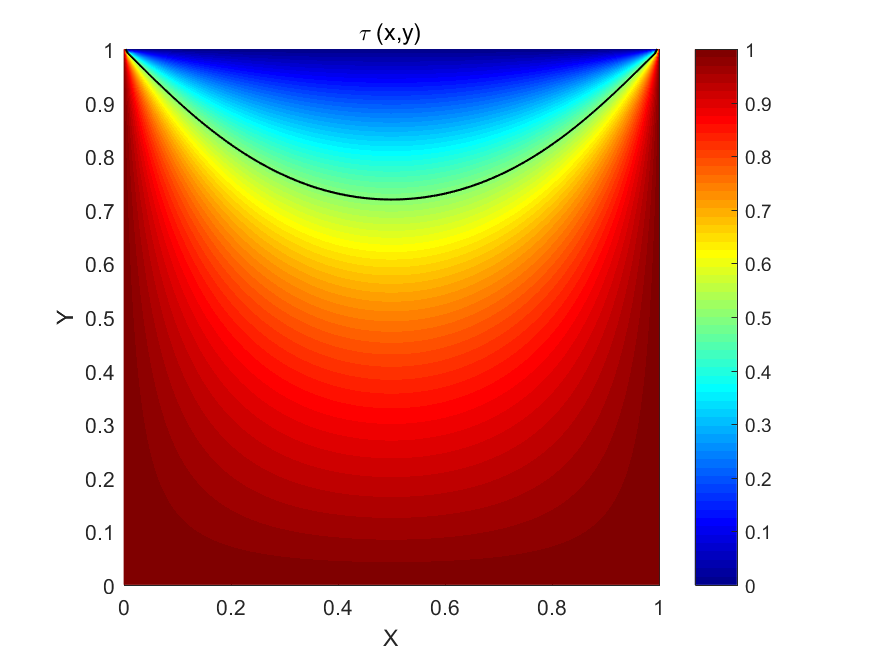}
		}
	\caption{Indicate function $\tau(x)$ with various characteristic boundaries. The black curve is the contour of $\tau(x)=0.5$.}
	\label{tau}
	\end{figure}

\section{Concluding remarks}

The Cauchy problem of the Laplace equation often appears in real applications,
and provides ways to infer global information from local measurement, which is also an ill-posed problem. 
The conditional stability estimates are proved, by which we can 
design stable numerical algorithms and estimate errors. The numerical treatment and the corresponding error estimate are presented. The estimate is featured by an indicate function constructed by the harmonic measure. This facilitates 
the evaluation of the numerical results, based on which how to improve the numerical results are proposed.
Although we treat only the Laplace equation in this work, similar results can be proved for more general elliptic equations. The Cauchy problem is closely related to the unique continuation problem, and for the latter,
conditional stability results and the numerical treatments are presented 
(e.g., \cite{Cheng1998}, \cite{Ke2022}).

\bigskip

{\bf Acknowledgement}\ \ This work was supported by the National Science Foundation of China 
(No. 11971121, No. 12201386) and Grant-in-Aid for Scientific Research (A) 20H00117 
of Japan Society for the Promotion of Science.

\iffalse
\addcontentsline{toc}{section}{
\fi

%\addcontentsline{toc}{chapter}{�ο�����}
%\bibliographystyle{elsarticle-num}%{unsrt}%{plain}
%\bibliographystyle{unsrt}
%\bibliographystyle{elsarticle-num}
%\bibliography{mybibfile-chen-pub-2}

\begin{thebibliography}{99}


\bibitem{Adams2003}Adams R.A. and Fournier J.J.F.,
Sobolev Spaces. Elsevier, Amsterdam, 2003.

\bibitem{Alessandrini2009}Alessandrini G., Rondi L., Rosset E., Vessella S.,
The stability for the Cauchy problem for elliptic equations. Inverse Problems, 
25 (2009) 123004.

\bibitem{Burman2018} Burman E., Hansbo P., and Larson, M., Solving ill-posed 
control problems by stabilized finite element methods: an alternative to Tikhonov regularization. Inverse Problems, 34(3) (2018) 035004.

\bibitem{Chakib2006}Chakib A. and Nachaoui A.,
Convergence analysis for finite element approximation to an inverse Cauchy 
problem. Inverse Problems 22 (2006) {\color{blue}1191-1206.}

\bibitem{Cheng1998} Cheng J. and Yamamoto M., Unique continuation on a line for harmonic functions. Inverse Probl. 14(4) (1998) {\color{blue}869-882}.
\bibitem{Cheng2001} Cheng J, Hon Y C, Wei T and Yamamoto M,
Numerical computation of a Cauchy problem for Laplace's equation Z. 
Angew. Math. Mech. 81 (2001) 665-674.

\bibitem{Salo2018} R\"uland A.,  Salo M.,
Quantitative runge approximation and inverse problems. 
International Mathematics Research Notices, 20 (2019) 6216-6234.

\bibitem{Cheng2000-1}Cheng J, Yamamoto M.,
One new strategy for a priori choice of regularizing parameters in Tikhonov's 
regularization. Inverse Problems, 16 (4) (2000) L31-L38.

\bibitem{Friedman1989} Friedman A. and Vogelius M., Determining cracks by boundary measurements. Indiana University Mathematics Journal, 38(3) (1989) 527-556.

\bibitem{Gilbarg1983}Gilbarg D. and Trudinger N. S.,
Elliptic Partial Differential Equations of Second Order. Springer, Berlin,
1983.

\bibitem{Hadamard1902} Hadamard, J., Sur les probl\`emes aux d\'eriv\'ees 
partielles et leur signification physique. Princeton University Bulletin, 
13 (1902) 49-52.

\bibitem{Hrycak2004}Hrycak T, and Isakov V.,
Increased stability in the continuation of solutions to the Helmholtz equation. Inverse Problems, 20(3) (2004) 697-712.

\bibitem{Isakov2006}Isakov V., Inverse Problems for Partial Differential 
Equations. Springer, Berlin, 2006.

\bibitem{Johansson2004} Johansson, T., An iterative procedure for solving a 
Cauchy problem for second order elliptic equations, Mathematische Nachrichten, 
272 (1) (2004) 46-54.

\bibitem{Ke2022}Ke Y, and Chen Y., Unique continuation on quadratic curves for harmonic functions. Chinese Annals of Mathematics. Series B. 43 (2022) 17-32.

\bibitem{Kellogg1953}Kellogg O. D., Foundations of Potential Theory. 
Dover Publications, Inc., New York, 1953.

\bibitem{Kozlov1990} V. A. Kozlov and V. G. Maz'ya, 
On iterative procedures for solving ill-posed boundary value problems that 
preserve differential equations, Algebra i Analiz 1 (1989) 144-170. 
English transl.: Leningrad Math. J. 1 (1990) 1207-1228.

\bibitem{Larsson2003} Larsson, S. and Thom\'e, V.,
 Partial Differential Equations with Numerical Methods, Springer, Berlin, 2003.

\bibitem{Lattes1969}R. Latt\`es and J.-L. Lions, The Method of Quasi-Reversibility, Applications to Partial Differential Equations,
American Elsevier Publishing Co., New York, 1969.

\bibitem{Lax1956} Lax, P.D., A stability theorem for solutions of abstract differential equations, and its application to the study of the local behavior of solutions of elliptic equations. Comm. Pure Appl. Math., 9(4) (1956) 747-766

\bibitem{Na}
Natterer, F.,
The finite element method for ill-posed problems. R.A.I.R.O. 
Analyse Num\'erique, 11(1) (1977) 271-278.

\bibitem{Payne1960}Payne, L. E., Bounds in the Cauchy problem for the Laplace equation. Archive for Rational Mechanics and Analysis, 5(1) (1960) 35-45.

\bibitem{Yang2005} Yang, X., Choulli M. and Cheng, J., 
An iterative method for the inverse problem of detecting corrosion in a pipe. 
Numerical Mathematics-A Journal of Chinese Universities (English Series), 
14(3) (2005) 252-266.
%}
\end{thebibliography}

\textbf{Reference}}
\end{document}